%% file: Draft-Jan2012.tex
\documentclass{article}

\input{macros}

\title{Enumeration schemes for vincular patterns}
\author{Andrew M. Baxter \and Lara K. Pudwell}

\begin{document}
\maketitle

\abstract{
We extend the notion of an enumeration scheme developed by Zeilberger and Vatter to the case of vincular patterns (also called ``generalized patterns" or ``dashed patterns").  In particular we provide an algorithm which takes in as input a set $B$ of vincular patterns and search parameters and returns a recurrence (called a ``scheme") to compute the number of permutations of length $n$ avoiding $B$ or confirmation that no such scheme exists within the search parameters.  We also prove that if $B$ contains only consecutive patterns and patterns of the form $\sigma_1\sigma_2 \dotsm \sigma_{t-1}\d\sigma_t$, then such a scheme must exist and provide the relevant search parameters.  The algorithms are implemented in Maple and we provide empirical data on the number of small pattern sets admitting schemes.  We make several conjectures on Wilf-classification based on this data.  We also outline how to refine schemes to compute the number of $B$-avoiding permutations of length $n$ with $k$ inversions.
}

\section{Introduction}

Enumeration schemes are special recurrences which were originally designed to compute the number of permutations avoiding a set of classical patterns.  In the current work we extend the tools of enumeration schemes to compute the number of permutations avoiding a set of \emph{vincular} patterns.

Let $[n]=\{1,2,\dotsc,n\}$.  For a word $w \in [n]^k$, we write  $w=w_1 w_2 \dotsm w_k$ and define the \emph{reduction} $\red(w)$ to be the word obtained by replacing the $i^{th}$ smallest letter(s) of $w$ with $i$.  For example $\red(839183)=324132$.  If $\red(u)=\red(w)$, we say that $u$ and $w$ are \emph{order-isomorphic} and write $u\oi w$.
 
Let $\S{n}$ be the set of permutations of length $n$.  We say that permutation $\pi\in\S{n}$ \emph{contains $\sigma\in\S{k}$ as a classical pattern} if there is some $k$-tuple $1\leq i_1 < i_2 <\dotsb <i_k\leq n$ such that $\red(\pi_{i_1} \pi_{i_2} \dotsm \pi_{i_k}) = \sigma$.  The subsequence $\pi_{i_1} \pi_{i_2} \dotsm \pi_{i_k}$ is called a \emph{copy} (or \emph{occurrence}) of $\sigma$.  If $\pi$ does not contain $\sigma$, then $\pi$ is said to \emph{avoid} $\sigma$.  Hence we see that $\pi=34512$ contains $231$ as a classical pattern witnessed by the subsequence $\pi_1 \pi_3 \pi_4 = 351$, but exhaustive checking shows $\pi$ avoids 132.  The subset of $\S{n}$ consisting of permutations avoiding $\sigma$ is denoted $\Sav{n}{\sigma}$.  For a set of patterns $B$, $\pi$ is said to avoid $B$ if $\pi$ avoids all $\sigma\in B$, and we denote the set of $B$-avoiding permutations by 
 \begin{equation}
  \Sav{n}{B} := \bigcap_{\sigma\in B} \Sav{n}{\sigma}.
 \end{equation}
We will denote the size of $\Sav{n}{B}$ by $\sav{n}{B} = \bigl| \Sav{n}{B} \bigr|$.

   Vincular patterns resemble classical patterns, with the constraint that some of the letters in a copy must be consecutive.  Formally, a \emph{vincular pattern of length $k$} is a pair $(\sigma, X)$ where $\sigma$ is a permutation in $\S{k}$ and $X\subseteq \{0,1,2,\dotsc, k\}$ is a set of ``adjacencies.''   A permutation $\pi\in\S{n}$ contains the vincular pattern $(\sigma, X)$ if there is a $k$-tuple $1\leq i_1 < i_2 < \dotsb < i_k\leq n$ such that the following three criteria are satisfied:    
   \begin{itemize}
     \item $\red(\pi_{i_1} \pi_{i_2} \dotsm \pi_{i_k}) = \sigma$.
     \item $i_{x+1} = i_x +1$ for each $x\in X\setminus\{0,k\}$.
     \item $i_1=1$ if $0\in X$ and $i_k=n$ if $k\in X$.
   \end{itemize}
In the present work we restrict our attention to patterns $(\sigma, X)$ where $\sigma\in\S{k}$ and $X\subseteq [k-1]$, rendering the third containment criterion moot.\footnote{We enact this restriction partly for simplicity.  It is plausible that the prefix-focused arguments below extend to patterns $(\sigma, X)$ with $0\in X$ with few modifications, but it is unlikely such an approach could work if $k \in X$.}
The subsequence $\pi_{i_1} \pi_{i_2} \dotsm \pi_{i_k}$ is called a \emph{copy} of $(\sigma, X)$.  In the permutation $\pi= 162534$, the subsequence $1253$ is a copy of $(1243, \{3\})$, but the subsequence $1254$ is not a copy since the 5 and 4 are not adjacent in $\pi$.  The classical pattern $\sigma$ is precisely the vincular pattern $(\sigma, \emptyset)$ since no adjacencies are required, while the consecutive pattern $\sigma$ is the vincular pattern $(\sigma, [k-1])$ since all internal adjacencies are required.
 
   In practice we write $(\sigma, X)$ as a permutation with a dash between $\sigma_j$ and $\sigma_{j+1}$ if $j\not\in X$.  Thus we will often refer to ``the vincular pattern $\sigma$'' without explicitly referring to $X$.  For example, $(1243, \{3\})$ is written $1\d2\d43$.    
   
   If the permutation $\pi$ does not contain a copy of the vincular pattern $\sigma$, then $\pi$ is said to \emph{avoid} $\sigma$.  We will use the same notation $\Sav{n}{\sigma}$ to denote the set of permutations avoiding the vincular pattern $\sigma$, and similarly $\Sav{n}{B}$ denotes those permutations avoiding every vincular pattern $\sigma\in B$.

  Observe that a vincular pattern $(\sigma, X)$ of length $k$ exhibits similar symmetries to those of permutations.  The reverse is given by $(\sigma,X)^{r} = (\sigma^{r}, k-X)$ where $k-X = \{k-x: x\in X\}$.  The complement is $(\sigma,X)^{c}=(\sigma^{c},X)$.  It follows that that $\pi$ avoids $\sigma$ if and only if $\pi^{r}$ avoids $\sigma^{r}$.  Similarly, $\pi$ avoids $\sigma$ if and only if $\pi^{c}$ avoids $\sigma^{c}$.  To consider inverses, one must generalize to the bivincular patterns introduced in \cite{Bousquet2010} which incorporate adjacency restrictions on not only the indices of letters forming a forbidden pattern, but on the values of the offending letters as well.  Since the inverse of a vincular pattern is not itself a vincular pattern, we will disregard inverses.
    
   Vincular patterns were introduced as ``generalized patterns'' by Babson and Steingr\'{i}msson in \cite{Babson2000} as a generalization of classical patterns as part of a systematic search for Mahonian permutation statistics.  They soon took on a life of their own spawning numerous papers, including \cite{Claesson2001, Kitaev2003, Kitaev2005, Elizalde2006}.  They have also been called ``dashed patterns" to distinguish them from other generalizations of classical patterns \cite{Chen2011}, but Claesson has since dubbed them ``vincular patterns'' to connect them with the bivincular patterns introduced in \cite{Bousquet2010}.  See Steingr\'{i}mssson's survey for a fuller history in \cite{Steingrimsson2010Survey}.  They have been linked to many of the common combinatorial structures such as the Catalan and Bell numbers as well several rarer or as-yet unseen structures. 
    
  The present work focuses on enumeration schemes, which were introduced by Zeilberger in \cite{Zeilberger1998} as an automated method to compute $\sav{n}{B}$ for many different $B$.  Vatter improved schemes in \cite{Vatter2008} with the introduction of gap vectors, and Zeilberger provided an alternate implementation in \cite{Zeilberger2006}.  The greatest feature of schemes is that they may be discovered automatically by a computer: the user need only input the set $B$ (along with bounds to the computer search) and the computer will return an enumeration scheme (if one exists within the bounds of the search) which computes $s_n(B)$ in polynomial time.  The second author extended these methods to consider pattern avoidance in permutations of a multiset in \cite{Pudwell2010a, Pudwell2008}, as well as barred-pattern avoidance in \cite{Pudwell2010b}.
    
Section \ref{sec:SchemeSummary} provides an overview of how enumeration schemes work and constructs a scheme for $23\d1$-avoiding permutations by hand.  Section \ref{sec:AutoDisc} outlines how the discovery of schemes can be done via a finite computer search.  Section \ref{sec:SpecialCases} demonstrates instances where we are guaranteed a successful search for a scheme.  Section \ref{sec:Known} is divided into three subsections.  The first provides an analysis of the algorithm's success rate in discovering schemes automatically, the second outlines the implications for Wilf-classification of vincular patterns, and the third gives an example of how enumeration schemes for vincular patterns may be adapted to count according to inversion number as per \cite{Baxter2010}.

\section{An Overview of Enumeration Schemes}\label{sec:SchemeSummary}

Broadly, enumeration schemes are succinct encodings for a system of recurrence relations to compute the cardinalities for a family of sets.  The enumerated sets are subsets of $\Sav{n}{B}$ determined by prefixes.  For a pattern $p\in \S{k}$, let $\Sp{n}{B}{p}$ be the set of permutations $\pi\in \Sav{n}{B}$ such that $\red(\pi_1 \pi_2 \ldots \pi_k)= p$.  For further refinement, let $w\in \{1,2,\dotsc,n\}^k$ and define the set 
\begin{equation*}
 \Spt{n}{B}{p}{w} = \{\pi\in \Sp{n}{B}{p}  \colon\, \pi_i = w_i \mathrm{\ for\ } 1\leq i\leq k\}.
\end{equation*}
 For an example, consider
\begin{equation*}
 \Spt{5}{1\d2\d3}{21}{53} = \{ 53142, 53214, 53241, 53412, 53421 \}.
\end{equation*}
Clearly this refinement is worthwhile only when $\red(w)=p$.  The redundancy of including $p$ in the $\Spt{n}{B}{p}{w}$ notation is maintained to emphasize the subset relation.  We will denote sizes of these sets by $\sp{n}{B}{p} = \bigl| \Sp{n}{B}{p} \bigr|$ and $\spt{n}{B}{p}{w} = \bigl| \Spt{n}{B}{p}{w} \bigr|$.

By looking at the prefix of a permutation, one can identify likely ``trouble spots'' where forbidden patterns may appear.  For example, suppose we wish to avoid the pattern $23\d1$.  Then the presence of the pattern $12$ in the prefix indicates the potential for the entire permutation to contain a $23\d1$ pattern.  In \cite{Vatter2008}, Vatter partitions $\Sav{n}{B}$ according to the inverse notion of the pattern formed by the \emph{smallest} $k$ letters in $\pi\in \Sav{n}{B}$.  This partition is not well-suited for keeping track of adjacencies.

Enumeration schemes take a divide-and-conquer approach to enumeration.  For a permutation $p\in\S{k}$, we say that $p'\in\S{k+1}$ is a \emph{child} of $p$ if $p'_1 p'_2\cdots p'_k \oi p$.  For example, the children of $p=312$ are $3124$, $4123$, $4132$, and  $4231$.  Any set $\Sp{n}{B}{p}$ for $p\in \S{k}$ may be partitioned into the family of sets $\Sp{n}{B}{p'}$ for each of the children $p' \in \Sp{k+1}{B}{p}$.  These smaller sets are then counted as described below, and their sizes are totaled to obtain $\sp{n}{B}{p}$.  In the end we will have computed $\sav{n}{B}$, since $\Sav{n}{B}= \Sp{n}{B}{\epsilon} = \Sp{n}{B}{1}$ for $n\geq 1$, where $\epsilon$ is the empty permutation.

For a prefix pattern $p\in \S{k}$, we will classify $\Sp{n}{B}{p}$ in one of three ways:

\begin{enumerate}
   \item[(1)] If $n=k$, then $\Sp{n}{B}{p}$ is either $\{p\}$ or $\emptyset$, depending on whether $p$ avoids $B$.  
   \item[(2)] For each $w \in [n]^k$ such that $\red(w)=p$ one of the following happens:
     \begin{itemize}
      \item[(2a)] $\Spt{n}{B}{p}{w}$ is empty, or
      \item[(2b)] $\Spt{n}{B}{p}{w}$ is in bijection with some $\Spt{{n'}}{B}{{p'}}{{w'}}$ for ${n'}<n$.
     \end{itemize}
   \item[(3)] $\Sp{n}{B}{p}$ must be partitioned further, so $\sp{n}{B}{p} = \sum\limits_{p'\in \Sp{k+1}{B}{p}} \sp{n}{B}{p'}$.
\end{enumerate}

Case (1) provides the base cases for our recurrence.  For case (2), if there is any $w$ for which neither (2a) nor (2b) holds then we must divide $\Sp{n}{B}{p}$ as in case (3).  For case 2a, the \emph{gap vector criteria} for the given $p$ identify which $w$ yield empty $\Spt{n}{B}{p}{w}$.  Gap vector criteria are developed in Subsection \ref{gv}.  The bijection in (2b) is performed by removing a certain subset of the first $k$ letters of $\pi\in \Spt{n}{B}{p}{w}$, and which subset may be ``nicely'' removed depends on $p$ and $B$ but not $w$.  Such subsets are called \emph{reversibly deletable}, and are developed in Subsection \ref{rd}.

%
%
 %

\subsection{Gap Vectors}\label{gv}

The motivation for gap vectors lies in the idea of ``vertical space'' (in the sense of the graph of a permuation) in a prefix $w$.  Sometimes the difference of the values of letters in the prefix is so great that a forbidden pattern \emph{must} appear.   To make this notion more precise, we follow our example above and compute $\sav{n}{B}$ for $B=\{23\d1\}$.  Observe that $\Spt{n}{B}{12}{w_1 w_2}$ is empty if $w_1> 1$, since if $\pi \in \Spt{n}{B}{12}{w_1 w_2}$ then $\pi_i=1$ for some $i\geq 3$.  Thus $\red(w_1 w_2 \pi_i)=231$ and so $\pi$ contains $23\d1$.  Hence $\Spt{n}{B}{12}{w_1 w_2}$ is non-empty only if $1=w_1<w_2\leq n$.

For a word $w\in [n]^{k}$, let $c_i$ be the $i^{th}$ smallest letter of $w$, $c_0=0$ and $c_{k+1}=n+1$.  Define the $(k+1)$-vector $\vec{g}(n,w)$ such that $i^{th}$ component $g_i = c_i - c_{i-1} -1$.  We call $\vec{g}(n,w)$ the \emph{spacing vector} of $w$. 
 Observe that for any $\pi\in \Spt{n}{B}{p}{w}$, $g_i$ is the number of letters in $\pi_{k+1}, \pi_{k+2}, \ldots, \pi_{n}$ which lie between $c_{i-1}$ and $c_i$, and so $\vec{g}(n,w)$ indicates what letters follow the prefix.  In the preceding paragraph, we saw that $\Spt{n}{B}{12}{w}$ is empty if $\vec{g}(n,w) \geq \langle 1,0,0 \rangle$ where $\geq$ represents the product order for $\mathbb{N}^3$ (component-wise dominance).  Towards generality, we make the following definition:

\begin{definition}
 Given a set of forbidden patterns $B$ and prefix $p\in\S{k}$, then $\vec{v}\in\mathbb{N}^{k+1}$ is a \emph{gap vector for prefix $p$ with respect to $B$} if for all $n$ $\Spt{n}{B}{p}{w}=\emptyset$ for all $w$ such that $\vec{g}(n,w)\geq \vec{v}$.  When this happens, we say that $w$ \emph{satisfies} the \emph{gap vector criterion} for $\vec{v}$.  
 \end{definition}

It should be noted that this definition reverses the terminology of \cite{Vatter2008} to match that of \cite{Zeilberger2006, Pudwell2010a, Pudwell2008, Pudwell2010b, Baxter2010}.


From this definition we see $\vec{v}=\langle 1,0,0 \rangle$ is a gap vector for $p=12$ with respect to $B=\{23\d1\}$, and any prefix $w=
w_1 w_2$ with $1<w_1<w_2\leq n$ satisfies the gap vector condition for $\vec{v}$.

Observe that gap vectors for a given prefix $p\in \S{k}$ form an upper order ideal in $\mathbb{N}^{k+1}$, i.e., if $\vec{u}\geq\vec{v}$ for gap vector $\vec{v}$ then $\vec{u}$ is also a gap vector.  Hence it suffices to determine only the minimal elements since they will form a basis (since $\mathbb{N}^{k+1}$ is partially well-ordered).  Details of the automated discovery of gap vectors are left to Section \ref{sec:AutoDiscGV}.

\subsection{Reversible Deletions}\label{rd}

If we are considering $\Spt{n}{B}{p}{w}$ for a $w$ that fails all gap vector criteria, we rely on bijections with previously-computed $\Spt{{n'}}{B}{{p'}}{{w'}}$ for ${n'}<n$.  To continue the example above, consider $\Spt{n}{23\d1}{12}{1w_2}$.  An initial $\pi_1=1$ cannot take part in a $23\d1$ pattern, so the map of deleting $\pi_1$ is a bijection 
\begin{equation*}
  d_1: \Spt{n}{23\d1}{12}{1w_2} \to \Spt{n-1}{23\d1}{1}{w_2-1}
\end{equation*}
where $d_1: \pi_1 \pi_2 \cdots \pi_n \mapsto \red(\pi_2 \pi_3 \cdots \pi_n)$.  Hence we see that $\spt{n}{23\d1}{12}{1w_2} = \spt{n-1}{23\d1}{1}{w_2-1}$.

More generally define the deletion $d_r(\pi) := \red(\pi_1 \ldots \pi_{r-1} \pi_{r+1} \ldots \pi_n)$, that is, the permutation obtained by omitting the $r^{th}$ letter of $\pi$ and reducing.  Furthermore for a set $R$, define $d_R(\pi)$ to be the permutation obtained by deleting $\pi_r$ for each $r\in R$ and then reducing.  For a word $w$ with no repeated letters, define $d_r(w)$ be the word obtained by deleting the $r^{th}$ letter and then subtracting 1 from each letter larger than $w_r$.  Similarly, to construct $d_R(w)$ delete $w_r$ for each $r\in R$ and subtract $\bigl|\{r\in R: w_r < w_i\}\bigr|$ from $w_i$.  For example $d_3(6348) =537$ and $d_{\{1,3\}}(6348) = 36$.  It can be seen that this definition is equivalent to the one given above when $w\in\S{k}$, and it allows for more succinct notation in the upcoming definition.  
For any set $R$ and $n\geq |R|$, ${d_R: \Spt{n}{\emptyset}{p}{w} \to \Spt{n-|R|}{\emptyset}{d_R(p)}{d_R(w)}}$ is a bijection.  Sometimes we are lucky and the restriction to $\Spt{n}{B}{p}{w}$ is a bijection with $\Spt{n-|R|}{B}{d_R(p)}{d_R(w)}$, leading to the following definition:

\begin{definition}
The index $r$ is \emph{reversibly deletable for $p$ with respect to $B$} if the map
 $$d_r: \Spt{n}{B}{p}{w} \to \Spt{n-1}{B}{d_r(p)}{d_r(w)} $$
is a bijection for all $w$ failing the gap vector criterion for every gap vector of $p$ with respect to $B$ (i.e., $d_r$ is a bijection for all $w$ such that $\Spt{n}{B}{p}{w}\neq \emptyset$).

The set of indices $R$ is \emph{reversibly deletable for $p$ with respect to $B$} if the map 
$$d_R: \Spt{n}{B}{p}{w} \to \Spt{n-|R|}{B}{d_R(p)}{d_R(w)} $$
is a bijection for all $w$ failing the gap vector criterion for every gap vector of $p$ with respect to $B$ (i.e., $d_R$ is a bijection for all $w$ such that $\Spt{n}{B}{p}{w}\neq \emptyset$).
\end{definition}

Note that the empty set $R=\emptyset$ is always reversibly deletable.  We are interested in finding non-empty reversibly deletable sets when they exist.  Also observe that if the prefix $p$ contains a forbidden pattern then $\Spt{n}{B}{p}{w}=\emptyset$ for any appropriate $w$, and so $\vec{0}=\langle 0,0,\ldots, 0 \rangle$ is a gap vector.  Furthermore if $\vec{0}$ is a gap vector then any set $R\subseteq \{1,2,\dotsc, |p|\}$ is vacuously reversibly deletable.

In \cite{Vatter2008} Vatter uses the term \emph{ES$^{+}$-reducible} to describe $p$ for which there is a non-empty reversibly deletable set.  When there is no non-empty reversibly deletable set, then $p$ is called \emph{ES$^{+}$-irreducible}.  We will not make use of this terminology in the current work.

In the classical case, Vatter proved that identifying reversibly deletable indices is a finite process in \cite{Vatter2008}.  We will prove the analogous result for vincular patterns in Section \ref{AutoDiscRD}.

\subsection{Formal Definition of an Enumeration Scheme}\label{sec:formal}

Formally, an enumeration scheme $E$ for $\Sav{n}{B}$ is a set of triples $(p, G_p, R_p)$, where $p\in \S{k}$ is a prefix pattern, $G_p$ is a basis of gap vectors for $p$ with respect to $B$, and $R_p$ is a reversibly deletable set for $p$ with respect to $B$.  Furthermore, $E$ must satisfy the following criteria:

\begin{enumerate}
 \item $(\epsilon, \emptyset, \emptyset)\in E$.
 \item For any $(p,G_p,R_p) \in E$,
  \begin{enumerate}
   \item If $R_p=\emptyset$ and $\vec{0}\notin G_p$, then $(p', G_{p'}, R_{p'}) \in E$ for every child $p'$ of $p$.
   \item If $R_p\neq\emptyset$, then $(\hat{p}, G_{\hat{p}}, R_{\hat{p}})\in E$ for $\hat{p}=d_{R_p}(p)$.
  \end{enumerate}
\end{enumerate}

One can then ``read'' the enumeration scheme $E$ to compute $\spt{n}{B}{p}{w}$ according to the following rules:

\begin{enumerate}
  \item If $w$ satisfies the gap vector criteria for some $\vec{v}\in G_p$, then $\spt{n}{B}{p}{w}=0$.
  \item For each prefix $w$ that fails the gap criteria for all $\vec{v}\in G_p$, $\spt{n}{B}{p}{w} = \spt{n-|R_p|}{B}{d_{R_p}(p)}{d_{R_p}(w)}$ (i.e., $R_p$ is a reversibly deletable set of indices).
  \item If $R_p=\emptyset$ then $\sp{n}{B}{p} = \sum\limits_{p'\in \Sp{k+1}{B}{p}} \sp{n}{B}{p'}$.
\end{enumerate}

When combined with the obvious initial condition that $\spt{n}{B}{p}{w}=1$ when $p$ has length $n$ and avoids $B$, the scheme provides a system of recurrences to compute $\spt{n}{B}{p}{w}$ and hence $\sav{n}{B}$.

As an example consider $\Sav{n}{23\d1}$, discussed above, with the enumeration scheme:
\begin{equation}\label{ES:23-1}
 E=\{(\epsilon, \emptyset, \emptyset), (1, \emptyset, \emptyset), (12, \{\langle 1, 0, 0 \rangle\}, \{1\}), (21, \emptyset, \{1\})\}.
\end{equation}
The definition of schemes implies $(\epsilon, \emptyset, \emptyset)\in E$, and  $R_\epsilon=\emptyset$ requires $(1, G_1, R_1)\in E$.  Starting with the pattern 1 yields no additional information so $G_1$ and $R_1$ are both empty.  Thus we see $(12, G_{12}, R_{12})\in E$ and $(21, G_{21}, R_{21})\in E$.  As discussed in Section \ref{gv}, $\langle 1,0,0 \rangle \in G_{12}$.  It is easily seen this forms a basis for all gap vectors for $12$ and so $G_{12}=\{\langle 1,0,0 \rangle\}$.  As discussed in Section \ref{rd}, $R_{12}=\{1\}$.

Moving on to the prefix pattern $p=21$, it can be seen that  $R_{21}=\{1\}$ by the following argument.  Suppose $\pi \in \Sp{n}{\emptyset}{21}$.  First observe that deleting $\pi_1$ cannot create a $23\d1$ which was not already present in $\pi_2 \cdots \pi_n$.  Next $\pi_1$ cannot take part in a $23\d1$ pattern since this would require $\red(\pi_1 \pi_2 ) = \red(23)=12$ while it is known that $\pi_1 > \pi_2$.  Hence the map $d_1$ restricts to a bijection $\Sp{n}{23\d1}{21}  \leftrightarrow \Sp{n-1}{23\d1}{1}$, so we may let $R_{21}=\{1\}$.  Since this argument can hold regardless of the actual letters $\pi_1$ and $\pi_2$, we may let $G_{21}=\emptyset$.  This completes the construction of $E$ above.

The scheme $E$ translates into the following system of recurrences:
\begin{equation*}
 \begin{split}
 \sav{n}{23\d1} &= \sp{n}{23\d1}{\epsilon} \\
               &= \sp{n}{23\d1}{1} \\
               &= \sum_{a=1}^{n} \spt{n}{23\d1}{1}{a} \\
\spt{n}{23\d1}{1}{a}  &= \sum_{b=1}^{a-1} \spt{n}{23\d1}{21}{ab} + \sum_{b=a+1}^{n} \spt{n}{23\d1}{12}{ab} \\
\spt{n}{23\d1}{21}{ab} &= \spt{n-1}{23\d1}{1}{b} \\
\spt{n}{23\d1}{12}{ab} &= 
 \begin{cases}
  \spt{n-1}{23\d1}{1}{b-1}, & a=1 \\
        0 ,                 & a>1
  \end{cases}
 \end{split}
\end{equation*}
               
This system simplifies to:
\begin{equation*}
  \spt{n}{23\d1}{1}{a} = 
  \begin{cases}
   \sum\limits_{b=1}^{n-1} \spt{n-1}{23\d1}{1}{b}, & a=1\\
   \sum\limits_{b=1}^{a-1} \spt{n-1}{23\d1}{1}{b}, & 1<a\leq n. \\
  \end{cases}
\end{equation*}
which can be used to compute arbitrarily many terms of the sequence $\sav{n}{23\d1}$ in polynomial time. 

Claesson shows in Proposition 3 of \cite{Claesson2001} that $\sav{n}{23\d1}$ is the $n^{th}$ Bell number, and his bijection also implies that $\spt{n}{23\d1}{1}{a}$ is the number of permutations of $[n]$ such that $a$ is the largest letter in the same block as 1.  The triangle formed is an augmented version of Aitken's array as described in OEIS sequence A095149 \cite{A095149}.

If $|E|$ is finite, we say that $B$ admits a finite enumeration scheme.  The length of the longest $p$ appearing is the called the \emph{depth} of $E$.  Not every set $B$ admits a finite enumeration scheme, the simplest example being the classical pattern $2\d3\d1$.  Let $E_{231}$ be the scheme for $\Sav{n}{2\d3\d1}$, and let $J_t = t(t-1)\cdots 21$ be the decreasing permutation of length $t$.  It can be shown that $G_{J_t}= \emptyset$ and $R_{J_t}=\emptyset$ for any $t$, and hence $E_{231}$ contains the triple $(J_t, \emptyset, \emptyset)$ for all $t\geq 1$ and hence is infinite.  It should be noted, however, that the enumeration scheme for $\Sav{n}{1\d3\d2}$ is finite (of depth 2) and $\sav{n}{2\d3\d1}=\sav{n}{1\d3\d2}$ by symmetry.

  In general, if $B$ admits an enumeration scheme $E_B$ of depth $d$ then its set of complements $B^c = \{\sigma^c: \sigma \in B\}$ also admits an enumeration scheme $E_{B^c}$ of depth $d$.  In fact, one can say $(p, R, G) \in E_B$ if and only if $(p^c, R, G^r)\in E_{B^c}$ where $G^r = \{ \langle g_{k+1}, g_{k}, \ldots, g_1\rangle: \langle g_1, g_2, \ldots, g_{k+1} \rangle \in G\}$.  This follows directly from the definitions given above and is left to the reader.  One cannot make analogous statements regarding $B^r = \{\sigma^r: \sigma\in B\}$, and so $B$ may not have a finite scheme while $B^r$ does.

\section{Automated Discovery}\label{sec:AutoDisc}

We now turn to the process of automating the discovery of enumeration schemes for vincular patterns, since this automation is the most outstanding feature of this method.  The overall algorithm proceeds as follows:

\begin{algorithm}\label{ESalg} \ \ 

\begin{enumerate}
  \item Initialize $E := \{ (\epsilon, \emptyset, \emptyset) \}$
  \item Let $P$ be the set of all children of all prefixes $p$ such that $(p,G_p,R_p)\in E$ and $R_p = \emptyset$ and $\vec{0}\notin G_p$.  If there are no such prefixes, return $E$.  Otherwise proceed to step 3.
  \item For each $p\in P$, find a basis of gap vectors $G_p$.
  \item For each $p\in P$, find a non-empty reversibly deletable set of indices $R_p$ given the gap vector criteria in $G_p$.  If no such $R_p$ exists, let $R_p=\emptyset$.
  \item Let $E= E \cup \{(p, G_p, R_p): p\in P\}$.
  \item Return to step 2.
\end{enumerate}
\end{algorithm}

Steps 1, 2, 5, and 6 are routine computations for a computer algebra system.  In the following subsections we present algorithms to automate steps 3 and 4.

\subsection{Gap Vectors}\label{sec:AutoDiscGV}

We first look at automating step 3 of Algorithm \ref{ESalg}.  As mentioned previously, the set of gap vectors forms an order ideal in $\mathbb{N}^{k+1}$ and therefore it suffices to find a finite basis of minimal gap vectors.  In this section we present a method to test whether a given $\vec{v}$ is a gap vector by checking finitely many permutations for pattern containment.  

The approach mimics that of \cite{Zeilberger2006}, rather than \cite{Vatter2008} where Vatter presents a stronger notion of gap vector which yields an \emph{a priori} bound on the set of vectors to check.  Define the norm of a vector $\vec{v}=\langle v_1, \dotsc, v_{k+1}\rangle$  to be the sum of its components, $|\vec{v}|= v_1 + \dotsb + v_{k+1}$.   In Vatter's notion of gap vector, if $\vec{v}$ is a basis gap vector then $|\vec{v}| \leq \max \{|\sigma|: \sigma \in B\}-1$.  In Zeilberger's method in \cite{Zeilberger2006}, the maximum norm of basis gap vectors is entered by the user as a parameter of the algorithm.

We choose the method of \cite{Zeilberger2006} since its implementation allows the user more control over runtime via a parameter that lets us set the maximum allowed gap norm.  This speeds computation time since it reduces the candidate pool for putative gap vectors, but this is at the cost of missing gap vectors which could make the enumeration scheme finite.  For example, it is shown in \cite{Vatter2008} that there is no finite enumeration scheme for the forbidden set $B=\{1\d4\d2\d3, 1\d4\d3\d2\}$ using only gap vectors of norm 1.    On the other hand, there is a depth 7 scheme for $1\d2\d3\d4\d5$-avoiding permutations which has maximum basis gap vector norm 1 instead of the a priori bound of 4.  A search for a depth 7 scheme with maximum gap vector norm 4 is impractical with the current implementation.  A search for a depth 7 scheme with maximum gap vector norm 1, however, completes in under five minutes with a finite scheme.
 
We now present a test for whether a specific vector $\vec{v}$ is a gap vector by checking finitely many cases.

  Given a set of forbidden patterns $B$, prefix $p\in\S{k}$, and vector $\vec{v}\in\mathbb{N}^{k+1}$,  define the set of permutations with prefix $p$ and spacing vector $\vec{v}$:
  \begin{equation*}
    A(p, \vec{v}) := \{\pi\in \S{|p|+|\vec{v}|}: \pi_1 \cdots \pi_k \oi p, \vec{g}(\pi_1 \cdots \pi_k) = \vec{v} \}.
  \end{equation*}
In this notation, $\vec{v}$ is a gap vector if $\vec{u}\geq \vec{v}$ implies that every $\pi\in A(p,\vec{u})$ contains some pattern in $B$.

Define the \emph{head} of a vincular pattern $(\sigma,X)$ to be the subpattern $(\red(\sigma_1\cdots\sigma_{\ell+1}), X)$ where $\ell=\max X$.  For example, the head of $(241652,\{2\})=2\d41\d6\d5\d3$ is $(\red(241), \{2\})=2\d31$.  The part of $\sigma$ following the head is a classical pattern, with dashes between every letter.

\begin{theorem}\label{thm:GapTest}
  Consider a prefix $p\in\S{k}$ and a spacing vector $\vec{v}\in\mathbb{N}^{k+1}$.  If every permutation $\pi\in A(p,\vec{v})$ contains a copy of some $\sigma\in B$ such that $\pi_1 \cdots \pi_k$ contains the head of the copy, then $\vec{v}$ is a gap vector.
\end{theorem}

\begin{proof}
We will demonstrate how to construct any permutation $\pi\in A(p,\vec{u})$ for $\vec{u}\geq \vec{v}$ from a $\pi'\in A(p,\vec{v})$ while preserving any copy of $\sigma \in B$ whose head lies in  $\pi'_1 \cdots \pi'_k$  .

  Let $\pi\in A(p,\vec{u})$ where $|p|=k$.  Let $c_i$ be the $i^{th}$ smallest letter in $\pi_1 \cdots \pi_k$ and let $c_0=0$ and $c_{k+1}=n+1$.  Define $C_i := \{\pi_j: j> k, c_{i-1}<\pi_j<c_i\}$ for $i\in [k+1]$, and observe that $u_i = \bigl| C_i \bigr|$.  For each $i$, choose $u_i - v_i$ letters of $C_i$, delete these letters from $\pi$, and reduce.  Note that the deleted letters all lie outside of the prefix $\pi_1 \cdots\pi_k$, so this process forms $\pi'\in A(p,\vec{v})$.  Reversing this process by re-inserting the letters provides the necessary construction of $\pi$ from $\pi'$.  By our hypothesis, $\pi'$ contains $\sigma\in B$ such that the head of $\sigma$ lies in the prefix $\pi'_1 \cdots \pi'_k$. Inserting letters after the prefix will not destroy this copy of $\sigma$ since the portion of $\sigma$ lying outside the head has no adjacency restrictions.  Hence $\pi\in A(p, \vec{u})$ also contains $\sigma$, and our result is proven.
\end{proof}

Note that $A(p, \vec{v})$ contains $|\vec{v}|!$ permutations, and each of these must be checked for $B$-containment.  Hence keeping $|\vec{v}|$ small is a significant computational advantage.

Note that the criterion that every permutation in $A(p,\vec{v})$ contains a copy of $\sigma\in B$ \emph{such that $\pi_1 \cdots \pi_k$ contains the head of the copy} is required.  For example, consider $B=\{124\d3, 123\d4\d5\}$.  Here $A(123,\langle 0,0,0,2\rangle) = \{12345, 12354 \}$, both of which contain a forbidden pattern although the copy of $124\d3$ contained in $12354$ does not have its head entirely in $p$.  Now observe that $234165\in A(123,\langle 1,0,0,2\rangle)$ avoids $B$ even though $\langle 1,0,0,2\rangle \geq \langle 0,0,0,2\rangle$: the inserted $1$ severs the copy of $124\d3$ without creating any other forbidden pattern.  

Note that Theorem \ref{thm:GapTest} provides a sufficient condition, but we do not prove necessity.  There may exist gap vectors $\vec{v}$ which do not satisfy the given criterion, but we have observed no such vectors in practice.

 In the computer implementation of this test, one must construct $A(p,\vec{v})$ explicitly.  This may be done by methods discussed in \cite{Zeilberger2006}.

As an example of applying Theorem \ref{thm:GapTest}, consider the $B=\{23\d1\}$-avoiding permutations, with prefix $p=12$ and suppose we search over all vectors with norm at most $2$.  Table \ref{tab:GapSearch} gives the relevant information for each of the ten candidates.

\begin{table}[htb]
 \centering
		\begin{tabular}{l|l|l}
			$\vec{v}$ & $A(12, \vec{v})$ & Gap vector? \\
			\hline
			$\langle 0,0,0 \rangle$ & $\{12\}$  & No \\
			$\langle 1,0,0 \rangle$ & $\{231\}$ & Yes \\
			$\langle 0,1,0 \rangle$ & $\{132\}$ & No \\
			$\langle 0,0,1 \rangle$ & $\{123\}$ & No \\
			$\langle 1,1,0 \rangle$ & $\{2413, 2431\}$ & Yes \\
			$\langle 1,0,1 \rangle$ & $\{2314, 2341\}$ & Yes \\
			$\langle 0,1,1 \rangle$ & $\{1324, 1342\}$ & No \\
      $\langle 2,0,0 \rangle$ & $\{3412, 3421\}$ & Yes \\
			$\langle 0,2,0 \rangle$ & $\{1423, 1432\}$ & No \\
			$\langle 0,0,2 \rangle$ & $\{1234, 1243\}$ & No
		\end{tabular}
		\caption{Computing gap vectors for $p=12$ with respect to $\{23\dcap1\}$}
		\label{tab:GapSearch}
\end{table}

Looking at the set of gap vectors determined $\{ \langle 1,0,0 \rangle, \langle 1,1,0\rangle, \langle 1,0,1\rangle, \langle 2,0,0\rangle \}$, we see the order ideal generated by these vectors has minimal basis $\{ \langle 1,0,0\rangle\}$.\footnote{The thoroughness of this example is perhaps misleading regarding the implementation.  Once the computer discovers that $\vec{v}$ is a gap vector, it need not bother testing any other $\vec{u}\geq\vec{v}$.}  Hence in the enumeration scheme we see $G_{12}=\{ \langle 1,0,0 \rangle\}$.

\subsection{Reversibly Deletable Sets}\label{AutoDiscRD}

We now turn our attention to automating step 4 of Algorithm \ref{ESalg}: discovering reversibly deletable sets of indices for a given prefix $p$.  Our scenarios-based approach parallels that of \cite{Zeilberger2006}.

Recall that for a set of indices $R$, the map $d_R$ deletes $\pi_r$ for each $r\in R$.  This forms a bijection $d_R: \Spt{n}{\emptyset}{p}{w} \to \Spt{n-|R|}{\emptyset}{d_R(p)}{d_R(w)}$, and when this map restricts to a bijection $\Spt{n}{B}{p}{w} \to \Spt{n-|R|}{B}{d_R(p)}{d_R(w)}$ we say that $R$ is \emph{reversibly deletable} for the prefix $p$.   In the classical case, the deletion of a letter or letters could not create a copy of a forbidden pattern.  For vincular patterns, however, deleting a letter may create the adjacency required to form an occurrence of a vincular pattern.  For example, $3142$ avoids $23\d1$ but $d_2(3142)=231$ does not since the 3 and 4 become adjacent to one another.  This does not preclude the existence of bijective maps $d_R$, it merely requires additional checks for the automated discovery.  In the end, a finite search for a reversibly deletable set suffices as in the classical case: it is only the manner in which we check each candidate which differs.  The need to check both directions of the map first appears in \cite{Pudwell2010b} when extending schemes for barred pattern avoidance.  The added twist needed for vincular pattern avoidance is the introduction of the ``null'' symbol $\,\n\,$.

Note $R$ is reversibly deletable when every $\pi\in\Spt{n}{\emptyset}{p}{w}$ avoids $B$ if and only $d_R(\pi)$ also avoids $B$.  Inversely, we could check whether every $\pi$ which contains some $\sigma\in B$ has image $d_R(\pi)$ which also contains some $\sigma'\in B$.  This approach was introduced by Zeilberger in \cite{Zeilberger2006} and used by the second author in \cite{Pudwell2010a, Pudwell2010b} when extending enumeration schemes to other contexts.

Let us illustrate the approach via an example before moving to the general case.  Consider $B = \{124\d3\}$ and prefix $p=132$.  We ask ``which letters of the prefix can participate in a $\sigma=124\d3$ pattern?''.  Suppose $\pi$ is a permutation with prefix pattern $132$, and that at least one letter of the prefix is part of a copy of $\sigma$.  If $\pi$ has minimal length, then $\pi$ must have the form $\red(132abc)$ where $a,b,c\in\mathbb{Q}$ such that $2abc\oi \sigma$: $\sigma$ starts with two rises and the descent $32$ in the prefix prevents a $\sigma$ from starting earlier.  There are four such permutations: $13\underline{246}\underline{5}$, $14\underline{236}\underline{5}$, $15\underline{236}\underline{4}$, $16\underline{235}\underline{4}$ (the occurence of $\sigma$ is underlined in each).  It will be necessary to keep track of where and when dashes in the contained copy of $\sigma$ appear outside of the prefix, which we denote with the ``null'' symbol $\,\n\,$.  This special character denotes the possibility for intervening letters but cannot participate in patterns itself.  Thus we write these four permutations as $13\underline{246}\n\underline{5}$,$14\underline{236}\n\underline{5}$, $15\underline{236}\n\underline{4}$, $16\underline{235}\n\underline{4}$.  Denote this set of \emph{containment scenarios} for $p=132$ by $A_{132}$.  We now apply $d_R$ for each $R\subseteq [3]$ and check whether the images under $d_R$ each contain $\sigma$.  This is done in Table \ref{ScenariosTable1}.  If $d_R(\pi)$ contains $\sigma$ for each $\pi\in A_{132}$, then $R$ passes the first test\footnote{In the classical case, this was the \emph{only} test.} for reverse deletability: the insertion $d_R^{-1}$ does not create any forbidden patterns in $B$ when applied to a permutation which already avoids $B$.  Looking across the rows of Table \ref{ScenariosTable1}, we see that $\{1\}$, $\{2\}$, and $\{1,2\}$ pass this test since every permutation in those rows contains $\sigma$.

\begin{table}[htb]
	\centering
		\begin{tabular}{r|cccc}
$\pi\in A_{132}$ & $13246\n5$ & $14236\n5$ & $15236\n4$ & $16235\n4$ \\
\hline
$d_{\{1\}}(\pi)$ & $2135\n4$ & $3125\n4$ & $4125\n3$ & $5124\n3$    \\
$d_{\{2\}}(\pi)$ & $1235\n4$ & $1235\n4$ & $1235\n4$ & $1235\n4$    \\
$d_{\{3\}}(\pi)$ & $1235\n4$ & $1325\n4$ & $1425\n3$ & $1524\n3$    \\
$d_{\{1,2\}}(\pi)$ & $124\n3$ & $124\n3$ & $124\n3$ & $124\n3$      \\
$d_{\{1,3\}}(\pi)$ & $124\n3$ & $214\n3$ & $314\n2$ & $413\n2$      \\
$d_{\{2,3\}}(\pi)$ & $124\n3$ & $124\n3$ & $134\n2$ & $143\n2$      \\
$d_{\{1,2,3\}}(\pi)$ & $13\n2$ & $13\n2$ & $13\n2$ & $13\n2$        \\
		\end{tabular}
	\caption{$d_R(\pi)$ for each $R\subseteq \{1,2,3\}$, $\pi\in A_{132}$.}
	\label{ScenariosTable1}
\end{table}

Since a deletion map creates new adjacencies and potentially a copy of $\sigma$, there is a second test that $R$ must pass to be reversibly deletable.  Consider $R=\{2\}$, which passed the previous test.  Applying $d_{\{2\}}$ to a permutation with prefix pattern $132$ will create a permutation with prefix patttern $12$, so we must consider the containment scenarios for the prefix $p=12$: $A_{12} = \{124\n3, 1235\n4\}$.  We then consider all ways each of these containment scenarios could have arisen by applying $d_{\{2\}}$ to a permutation with prefix patern $132$; i.e., every permutation of the form $\red(1a24\n3)$ for $a\in\{2+\tfrac{1}{2},3+\tfrac{1}{2}, 4+\tfrac{1}{2}\}$ or $\red(1b235\n4)$ for $b\in\{2+\tfrac{1}{2}, 3+\tfrac{1}{2}, 4+\tfrac{1}{2}, 5+\tfrac{1}{2} \}$.  In particular, this list includes $1325\n4$, which avoids $\sigma$ while $d_{\{2\}}(1325\n4)=124\n3$ contains $\sigma$.  Since one can use $d_{\{2\}}$ to create a $\sigma$-containing permutation from a $\sigma$-avoiding permutation, $R=\{2\}$ cannot be reversibly deletable.  On the other hand, for $R=\{1\}$ one can check that the containment scenarios for $d_{\{1\}}(132)=21$ are $A_{21}=\{{2135\n4}, {3125\n4}, {4125\n3}, {5124\n3} \}$ and that the permutations starting with $132$ which map to some $\pi\in A_{21}$ are precisely $\{ {13246\n5}, {14236\n5}, {15236\n4}, {16235\n4} \}$.
Since each of these pre-image permutations contains $\sigma$, $R=\{1\}$ passes the second test for reversible deletability.  Hence $\{1\}$ is reversibly deletable.  Similarly, for $R=\{1,2\}$ we get only the containment scenario $A_1=\{124\n3\}$ and the same set of pre-images with prefix $132$:
\begin{equation*}
 d_{\{1,2\}} \left( \{13246\n5, 14236\n5, 15236\n4, 16235\n4 \} \right) = A_1.
\end{equation*}
 Again, each of the permutations on the lefthand side contains $\sigma$, so $R=\{1,2\}$ is reversibly deletable.  Hence we have two non-empty reversibly deletable sets for prefix $132$.  While either set will lead to a valid enumeration scheme, we follow a convention to choose the largest one and break ties lexicographically by the smallest elements.

To demonstrate a subtlety of containment scenarios, consider the forbidden set $B=\{3\d21, 32\d1\}$ and prefix $p=21$.  Here we see that we have the basis gap vector $\langle 1,0,0\rangle$, and so any permutation starting with prefix word $ab$ for $a>b>1$ necessarily contains a forbidden pattern.  Hence to prove $R$ is reversibly deletable, we only need to show $d_R$ is bijective starting from sets of the form $\Spt{n}{B}{21}{a1}$.  Therefore even though $42\n31$ contains a forbidden pattern and begins with $21$, we know that $\Spt{n}{B}{21}{42}=\emptyset$ and so we do not need to check whether $d_R(42\n31)$ contains a forbidden pattern.  In fact the only containment scenario worth checking for $p=21$ is $41\n32$.  Hence $R=\{2\}$ passes the first test for reversible deletability.  We then move on to consider the containment scenarios for prefix pattern $d_2(21)=1$.  These are $A_1 = \{3\n21, 32\n1\}$.  The pre-images under $d_2$ starting with $21$ include $413\d2$, however, which does not contain either forbidden pattern.  Hence $R=\{2\}$ fails the second test for reversible deletability.  If we had not kept track of dashes with the null character $\,\n\,$, however, the preimage $4132$ would have contained a forbidden pattern and $\{2\}$ would have appeared to be reversibly deletable.

We now outline in general the scenarios method to test whether a set $R$ is reversibly deletable for prefix $p$ with respect to forbidden pattern $B$.  We begin with a formal definition for a containment scenario.

\begin{definition}
 Let $(\sigma, X)$ be a vincular pattern of length $\ell$ and let $p\in\S{k}$ be a prefix pattern with known set of gap vectors $G$.  Let $w\in \{1,2,\ldots, n, \,\n\,\}^{n+m}$ be a word with $m$ copies of $\,\n\,$ and no other letters repeated.  Then $w$ is a \emph{containment scenario} for $p$ if the following criteria are satisfied:
 \begin{enumerate}
  \item $w_1 \cdots w_k \oi p$.  Note this implies $\n$ does not appear in the first $k$ letters.
  \item $w_1 \cdots w_k$ fails all gap vector criteria in $G$.
  \item There is some subsequence $1\leq i_1 < \cdots < i_{\ell} \leq n+m$ such that $w(i_1) \cdots w(i_{\ell}) \oi \sigma$ and $w(i_{x}+1)=\,\n\,$ for each $x$ such that $i_x\geq k$ and there is a dash between $\sigma_{x}$ and $\sigma_{x+1}$ (i.e., $x\not\in X$).
  \item No subsequence of $w$ is a containment scenario (i.e., $w$ has minimal length).
 \end{enumerate}
\end{definition}

The set of containment scenarios for a forbidden set $B$ is simply the union of the sets of containment scenarios for each $\sigma\in B$.  We will denote by $A_p$ the set of containment scenarios for a forbidden set $B$, a prefix $p$, and a set of gap vectors $G$.

One can compute $A_p$ via brute force over all $2^{|p|}-1$ nonempty subsequences of $p$.  A set of indices $1\leq i_1 < \cdots i_{t} \leq |p|$ is a \emph{partial match} for $(\sigma,X)\in B$ if $i_x+1 = i_{x+1}$ for each $x\in X$ and $p(i_1) p(i_2) \cdots p(i_{t}) \oi \sigma_1 \sigma_2 \cdots \sigma_{t}$.  Note that a set of indices may be a partial match for more than one pattern in $B$.  For each partial match of $(\sigma, X)$, insert the $|\sigma|-t$ letters and necessary number of $\n$ on the right end of $p$ in such a way to complete the occurrence of $\sigma$ using the letters in the partial match.  Repeating this process for each $(\sigma, X)\in B$ gives us the complete set of containment scenarios.  We may then throw out any containment scenarios whose first $|p|$ letters satisfy a gap vector criterion for some basis gap vector.

We now present the algorithm to check whether a given set $R\subseteq [k]$ is reversibly deletable for a prefix $p\in\S{k}$ with respect to the forbidden set $B$.

\begin{algorithm} \ \ 

 \begin{enumerate}
  \item Compute the set of containment scenarios $A_p$.
  \item For each $\pi\in A_p$, check if $d_R(\pi)$ contains a forbidden pattern in $B$.  If any $\pi$ avoids $B$, then $R$ is not reversibly deletable.  Otherwise, proceed to step 3.
  \item Compute the set of containment scenarios $A_{d_R(p)}$.
  \item Find the set of all permutations $\pi$ with prefix $p$ such that $d_R(\pi)\in A_{d_R(p)}$.  If any of these $\pi$ avoids $B$, then $R$ is not reversibly deletable.  If each of these contains some forbidden pattern, then $R$ is reversibly deletable.
 \end{enumerate}
\end{algorithm}


Therefore we can compute non-empty reversibly deletable sets automatically by a finite computer search.  This concludes our discussion on automated discovery of enumeration schemes.  These procedures have been implemented in the Maple package \textsc{gVatter}, available on the authors' homepages.

\section{Special Cases of Guaranteed Success}\label{sec:SpecialCases}

Knowing \textit{a priori} whether a set of patterns $B$ has a finite enumeration scheme remains an open question.  As a partial classification, we show here that if $B$ contains only consecutive (i.e., dashless) patterns and patterns of the form $\sigma_1 \sigma_2 \cdots \sigma_{t-1} \d \sigma_{t} = (\sigma, [t-2])$ then $B$ admits a finite enumeration scheme.  We will prove this via a series of lemmas regarding gap vectors and reversibly deletable sets.

Our first sequence of lemmas regards consecutive pattern avoidance.
\begin{lemma}\label{lem:cons1}
  Let $\sigma\in\S{t}$ and consider prefix pattern $p\in\S{k}$.  Then the set $\{1\}$ is reversibly deletable for prefix $p$ with respect to $\{(\sigma,[t-1])\}$ if $p_1 p_2 \cdots p_{\min(t,k)} \not\oi \sigma_1 \sigma_2 \cdots \sigma_{\min(t,k)}$.
\end{lemma}

\begin{proof}
  Suppose the permutation starts with prefix pattern $p$ such that $\pi$ $p_1 p_2 \cdots p_{\min(t,k)} \not\oi \sigma_1 \sigma_2 \cdots \sigma_{\min(t,k)}$.  Then $\pi_1$ could not be involved in a copy of the consecutive pattern $(\sigma, [t-1])$, since otherwise $p_1 p_2 \cdots p_{\min(t,k)} \oi \sigma_1 \sigma_2 \cdots \sigma_{\min(t,k)}$.  To verify $\{1\}$ is reversibly deletable, observe that applying the deletion map $d_1$ does not create new adjacencies, and so $d_1(\pi)$ avoids $(\sigma,[t-1])$ if and only if $\pi$ avoids $(\sigma,[t-1])$.
\end{proof}

Lemma \ref{lem:cons1} is conservative since larger sets can be reversibly deletable.  The set $\{1,2,\ldots, s\}$ is reversibly deletable for prefix $p$ with respect to $(\sigma, [t-1])$ if
$p_a p_{a+1} \cdots p_{\min(a+t-1, |p|)}$ is not order-isomorphic to $\sigma_{1} \sigma_{2} \cdots \sigma_{\min(t, |p|-a+1)}$ for each $a \leq s$.  This fact can be proven by the same arguments, since $d_{\{1,2,\ldots,s\}}$ will not create new adjacencies.  Looking ahead to Lemma \ref{lem:unionRD}, we will keep to $s=1$ so that the only non-empty reversibly deletable set which appears in a given scheme is $\{1\}$.
 
\begin{lemma}\label{lem:cons2}
    Let $\sigma\in\S{t}$.  For any permutation $p$ containing the consecutive pattern $(\sigma, [t-1])$, $\vec{0} = \langle 0,0,\ldots, 0\rangle$ is a gap vector for prefix $p$ with respect to $\{(\sigma,[t-1])\}$.
\end{lemma}

\begin{proof}
If $p$ already contains $\sigma$, then no permutation containing $p$ (in particular, starting with prefix pattern $p$) can avoid $\sigma$.  Hence $\Spt{n}{(\sigma,[t-1])}{p}{w}=\emptyset$ for any prefix word $w$, and the result follows from the definition of gap vectors.
\end{proof}

Lemmas \ref{lem:cons1} and \ref{lem:cons2} combine to give us the following Proposition:
\begin{proposition}\label{prop:cons}
  For $\sigma\in\S{t}$, the pattern set $\{(\sigma, [t-1])\}$ admits an enumeration scheme of depth $t$ where every reversibly deletable set is either $\emptyset$ or $\{1\}$.
\end{proposition}

\begin{proof}
 We will construct an enumeration scheme $E$ for $(\sigma,[t-1])$-avoiding permutations.  For each permutation $p$ of length at most $t$, Lemmas \ref{lem:cons1} and \ref{lem:cons2} imply a gap vector basis $G_p$ and a reversibly deletable set $R_p$, and we add the triple $(p, G_p, R_p)$ to $E$.  Observe every permutation of length exactly $t$ is either $p=\sigma$, in which case $\vec{0}$ is a gap vector, or $p\neq\sigma$, in which case $\{1\}$ is reversibly deletable, so we see that $E$ satisfies criterion (2a) given in Section \ref{sec:formal}.  Since every permutation $p$ of length at most $t$ appears in $E$, we see that criterion (2b) is also satisfied.  Hence $E$ is a valid enumeration scheme with finitely many elements.
\end{proof}

This proof has some redundancy built into it.  For a more efficient approach, we could execute Algorithm \ref{ESalg}, using Lemmas \ref{lem:cons1} and \ref{lem:cons2} to obtain gap vector bases (either $\emptyset$ or $\{ \vec{0} \}$) and our reversibly deletable sets (either $\emptyset$ or $\{1\}$).  The algorithm terminates since every prefix pattern $p$ of length $t$ is either $p=\sigma$, in which case $\vec{0}$ is a gap vector, or $p\neq\sigma$, in which case $\{1\}$ is reversibly deletable.  Hence $E$ contains no prefix patterns of length greater than $t$.

We now present the analogues of Lemmas \ref{lem:cons1} and \ref{lem:cons2} and Proposition \ref{prop:cons} for patterns of the form $(\sigma, [t-2])$ for $\sigma\in\S{t}$.

\begin{lemma}\label{lem:tail1}
  Let $\sigma\in\S{t}$ and consider prefix pattern $p\in\S{k}$.  Then the set $\{1\}$ is reversibly deletable for prefix $p$ with respect to $(\sigma,[t-2])$ if $p_1 p_2 \cdots p_{\min(t-1,k)} \not\oi \sigma_1 \sigma_2 \cdots \sigma_{\min(t-1,k)}$.
\end{lemma}

\begin{proof}
 As in the proof for \ref{lem:cons1}, the letter $\pi_1$ cannot be involved in the pattern $(\sigma, [t-2])$ unless $\pi_1 \pi_2 \cdots \pi_{\min(t-1,k)}$ is order-isomorphic to $\sigma_1 \sigma_2 \cdots \sigma_{\min(t-1,k)}$.  Furthermore, the deletion map $d_1$ does not create new adjacencies and so preserves $(\sigma, [t-2])$-avoidance. 
\end{proof}

Like Lemma \ref{lem:cons1}, Lemma \ref{lem:tail1} is conservative.  The set $\{1,2,\ldots, s\}$ is reversibly deletable for prefix $p$ with respect to $(\sigma, [t-2])$ if
$p_a p_{a+1} \cdots p_{\min(a+t-2, |p|)}$ is not order-isomorphic to $\sigma_{1} \sigma_{2} \cdots \sigma_{\min(t-1, |p|-a+1)}$ for each $a \leq s$.  

\begin{lemma}\label{lem:tail2}
    Let $\sigma\in\S{t}$ and consider the prefix pattern $p=\sigma_1 \cdots \sigma_{t-1}$.  Then $\vec{v} = \langle 0,\ldots,0,1,0,\ldots, 0\rangle$ is a gap vector for the prefix $p$ with respect to $\{(\sigma,[t-2])\}$, where $v_{\sigma_t}=1$ and $v_i=0$ for other $i$.  Further, $\{1\}$ is reversibly deletable for $p$ with respect to $(\sigma,[t-2])$.
\end{lemma}

\begin{proof}
 Let $\pi\in A(p,\vec{u})$ for $\vec{u}\geq \vec{v}$.  Then we know $\pi_1\cdots \pi_{t-1} \oi \sigma_{1} \cdots \sigma_{t-1}$, so define indices $a$ and $b$ so that $\pi_a$ corresponds to $\sigma_{t}-1$ and $\pi_b$ corresponds to $\sigma_{t}+1$.  Since $\pi$ has spacing vector $\vec{u}\geq\vec{v}$, we know there is some $\pi_i$ for $i>t-1$ such that the value $\pi_i$ lies between the values of $\pi_a$ and $\pi_b$.  In short, $\pi_1 \cdots \pi_{t-1} \pi_i$ forms a copy of $(\sigma,[t-2])$.  Hence $\vec{v}$ is a gap vector as per Theorem \ref{thm:GapTest}.

 We now turn to proving $\{1\}$ is reversibly deletable.  Suppose that $\pi\in\Spt{n}{(\sigma,[t-2])}{p}{w}$ for a prefix word $w$ such that $\vec{u}=\vec{g}(n,w)\not\geq v$.  Then $u_{\sigma_t}=0$, so we see that there is no $\pi_i$ for $i>t-1$ such that $\pi_1 \cdots \pi_{t-1} \pi_i \oi \sigma$.  Therefore $\pi_1$ cannot be involved in any copies of $(\sigma_1 \sigma_2\cdots\sigma_{t-1}, [t-2])$ and hence cannot be involved in a copy of $(\sigma, [t-2])$.  Since $d_1$ preserves $(\sigma, [t-2])$-avoidance as previously seen, we have shown that $\{1\}$ is reversibly deletable.
\end{proof}

We combine Lemmas \ref{lem:tail1} and \ref{lem:tail2} to form Proposition \ref{prop:tail}:

\begin{proposition}\label{prop:tail}
  For $\sigma\in\S{t}$, the pattern set $\{(\sigma, [t-2])\}$ admits an enumeration scheme of depth $t-1$ where every reversibly deletable set is either $\emptyset$ or $\{1\}$.
\end{proposition}

\begin{proof}
  The proof is similar to that of Proposition \ref{prop:cons}.  Construct $E$ by adding the triple $(p,G_p,R_p)$ for each permutation $p$ of length at most $t-1$, where the gap vector basis $G_p$ and reversibly deletable set $R_p$ are given by Lemmas \ref{lem:tail1} and \ref{lem:tail2}.  Since every permutation $p$ of length $t-1$ has $R_p=\{1\}$, we see that $E$ is finite while still satisfying criterion (2a).
\end{proof}

As in Proposition \ref{prop:cons}, the scheme constructed in Proposition \ref{prop:tail} has more terms than the scheme which would be constructed via Algorithm \ref{ESalg}.

Propositions \ref{prop:cons} and \ref{prop:tail} demonstrate how to construct schemes for singleton sets of patterns of certain forms.  The following lemmas outline when one can combine enumeration schemes for pattern sets $B$ and $B'$ to construct an enumeration scheme for $B \cup B'$.

\begin{lemma}\label{lem:unionGV}
  Suppose that $\vec{v}$ is a gap vector for the prefix $p$ with respect to $B$.  Then $\vec{v}$ is a gap vector for $p$ with respect to any pattern set $C\supseteq B$.  In particular, for a pattern set $B'$, $\vec{v}$ is a gap vector for $C=B \cup B'$.
\end{lemma}

\begin{proof}
  If $C\supseteq B$, then $\Spt{n}{C}{p}{w} \subseteq \Spt{n}{B}{p}{w}$.  Therefore any criterion implying $\Spt{n}{B}{p}{w}=\emptyset$, in particular a gap vector criterion, also implies that $\Spt{n}{C}{p}{w}=\emptyset$.
\end{proof}

\begin{lemma}\label{lem:unionRD}
  Suppose that $R$ is reversibly deletable for the prefix $p$ with respect to $B$ and with respect to $B'$.  Then $R$ is reversibly deletable for $p$ with respect to $B \cup B'$.
\end{lemma}

\begin{proof}
 First recall that $\Spt{n}{B \cup B'}{p}{w} = \Spt{n}{B}{p}{w} \cap \Spt{n}{B'}{p}{w}$.  From the definition of a reversibly deletable set, we know that we have the following bijections:
  \begin{equation*}
   \begin{split}
     d_R:& \Spt{n}{B}{p}{w} \to \Spt{n-|R|}{B}{d_R(p)}{d_R(w)} \\
     d_R:& \Spt{n}{B'}{p}{w} \to \Spt{n-|R|}{B'}{d_R(p)}{d_R(w)} \\
   \end{split}
  \end{equation*}
  It follows that we also have the bijection
  \begin{equation*}
    d_R: \Spt{n}{B \cup B'}{p}{w} \to \Spt{n-|R|}{B \cup B'}{d_R(p)}{d_R(w)}.
  \end{equation*}
 Hence $R$ is reversibly deletable for $p$ with respect to $B\cup B'$.
\end{proof}

We may combine Lemmas \ref{lem:unionGV} and \ref{lem:unionRD} combine with Propositions \ref{prop:cons} and \ref{prop:tail} to give us the following theorem:

\begin{theorem}\label{thm:BigClass}
 If a finite set $B$ contains only patterns of the form $(\sigma, [|\sigma|-1])$ and $(\sigma, [|\sigma|-2])$, then $B$ admits a finite enumeration scheme where every reversibly deletable set is either $\emptyset$ or $\{1\}$.
\end{theorem}

\begin{proof}
 

Let $M$ be the maximum length of patterns $(\sigma, X)\in B$.  For each permutation $p$ of length at most $M$, let $G_p(\sigma, X)$ be the gap vector basis for prefix pattern $p$ with respect to $(\sigma, X)$ as implied by Lemma \ref{lem:cons2} or \ref{lem:tail2}.  Define the set
\begin{equation}
 G_p := \bigcup_{(\sigma, X)\in B} G_p(\sigma, X).
\end{equation}
By Lemma \ref{lem:unionGV}, $G_p$ is a gap vector basis for $p$ with respect to $B$.  Observe that $G_p$ may not be a minimal basis.

Similarly let $R_p(\sigma, X)$ be the reversibly deletable set for $p$ with respect to $(\sigma, X)$ as implied by Lemma \ref{lem:cons1}, \ref{lem:tail1}, or \ref{lem:tail2}.  Define the set
\begin{equation}
 R_p := \bigcap_{(\sigma, X)\in B} R_p(\sigma, X).
\end{equation}
Observe that each $R_p(\sigma, X)$ is either $\emptyset$ or $\{1\}$, so $R_p$ is either $\emptyset$ or $\{1\}$.  Further, $R_p=\{1\}$ if and only if each $R_p(\sigma, X)=\{1\}$.  Hence Lemma \ref{lem:unionRD} implies $R_p$ is reversibly deletable for $p$ with respect to $B$.

Let $E$ be the set of triples $(p, G_p, R_p)$ for each permutation $p$ of length at most $M$.  Clearly $E$ is finite and satisfies criteria (1) and (2b), and $E$ satisfies (2a) for any prefix $p$ of length less than $M$.  It remains to show that any prefix $p$ of length $M$ has either $\vec{0}\in G_p$ or non-empty $R_p$.   If $p$ contains any forbidden pattern $(\sigma, X)\in B$, then $G_p(\sigma,X) = \{\vec{0}\}$ by Lemma \ref{lem:cons2} or \ref{lem:tail2}.  Hence $\vec{0}\in G_p$.  On the other hand if $p$ avoids $B$, then $p$ avoids each $(\sigma, X)\in B$ and so each $R_p(\sigma, X) = \{1\}$ by Lemma \ref{lem:cons1}, \ref{lem:tail1}, or \ref{lem:tail2}.  Hence $R_p = \{1\}$.  Thus we see that $E$ satisfies criterion (2a) for all $p$ of length $M$, and so $E$ is a valid enumeration scheme for $B$.
\end{proof}

It is worth mentioning that Proposition \ref{prop:tail} and its supporting lemmas can be generalized to obtain a finite scheme for any set of patterns of the form $B=\{\sigma_1\sigma_2\cdots\sigma_{t}\d\sigma_{t+1}\d\sigma_{t+2}, \sigma_1\sigma_2\cdots\sigma_{t}\d\sigma_{t+2}\d\sigma_{t+1}\}$ based on the enumeration scheme for their common head $\tau = \red(\sigma_1\sigma_2\cdots\sigma_{t})$.  In this case the prefix $p=\tau$ has a gap vector basis $G_p = \{\vec{v}\}$, where $v_{\sigma_{t+1}} = 1$, $v_{\sigma_{t+2}-1}=1$ and $v_i = 0$ for other $i$ for $\sigma_{t+1}<\sigma_{t+2}-1$.  If $\sigma_{t+1}=\sigma_{t+2}-1$, then $v_{\sigma_{t+1}} = v_{\sigma_{t+2}-1}=2$ instead.  The triples for other prefixes $p\neq\tau$ appearing in $E_{\tau}$ transfer unchanged to $E_{\sigma}$.

One may similarly construct a scheme for the set of $k!$ patterns formed by appending to the consecutive pattern $\tau$ a dashed tail of $k$ letters $\sigma_{t+1}, \sigma_{t+1}+1, \ldots, \sigma_{t+k}$ in all possible orderings, e.g., $B=\{12\d3\d4\d5$, $12\d3\d5\d4$, $12\d4\d3\d5$, $12\d4\d5\d3$, $12\d5\d3\d4$, $12\d5\d4\d3\}$.  If we suppose $\sigma_{t+j} < \sigma_{t+j+1}$, then let $\vec{u}^{(j)}$ be the $0$-$1$ vector with a $1$ in position $\sigma_{t+j}-(j-1)$.  Then for prefix $p=\tau$ we have the gap vector $\vec{v}=\sum_{j} \vec{u}^{(j)}$.  The remaining triples $(p,G_p,R_p)\in E_{\tau}$ transfer unchanged to $E_{\sigma}$.  These pattern sets are identical to Kitaev's notion of partially-ordered generalized patterns in \cite{Kitaev2005}, where some letters of the pattern are incomparable (or rather, do not need to be compared).  Thus the example $B$ would be written as a single such pattern $12\d3\d3'\d3''$ where the letters acting as $3$, $3'$, $3''$ are incomparable.  The arguments above imply that the following is an enumeration scheme for patterns avoiding $12\d3\d3'\d3''$:
\begin{equation}
 \bigl\{ (\epsilon, \emptyset, \emptyset), (1, \emptyset, \emptyset), 
  (12, \{\langle 0,0,3 \rangle\}, \{1\}), (21, \emptyset, \{1\}) \bigr\}.
\end{equation}

One might hope we can continue this trend of adding vincular portions to patterns with known schemes to get new schemes, but of course this does not work in general\footnote{If it did then we would get enumeration schemes for all classical patterns, which certainly is not the case.}.  Still, there may be some interesting relationships between two vincular patterns with the same underlying permutation.  For example, every pattern $(1234, X)$ for $X\subseteq\{1,2,3\}$ has a finite scheme, whose depths (based on the implementation in gVATTER) are summarized in Table \ref{1234patterns}.  There do not appear to be any clear patterns dictating scheme depth for the vincular pattern $(1234,X)$ given the depths of other patterns $(1234,X')$, based on subset relations between $X$ and $X'$.  

\begin{table}[h]
	\centering
		\begin{tabular}{ccl}
	 $\sigma$     &  $X$         &  Scheme Depth  \\
			\hline
     $ 1234$    &  $\{1,2,3\}$ & 4   \\
    $ 123\d4$   &  $\{1,2\}$   & 3   \\
    $ 12\d34$   &  $\{1,3\}$   & 4   \\
    $ 1\d234$   &  $\{2,3\}$   & 4   \\
   $ 12\d3\d4$  &  $\{1\}$     & 4   \\
   $ 1\d23\d4$  &  $\{2\}$     & 3   \\
   $ 1\d2\d34$  &  $\{3\}$     & 5   \\
  $ 1\d2\d3\d4$ &  $\emptyset$ & 4   \\
		\end{tabular}
	\caption{Scheme depth for vincular pattern $(1234, X)$}
	\label{1234patterns}	
\end{table}

\section{Known successes and applications}\label{sec:Known}

\subsection{Analysis of Success Rates}\label{sec:Success}

Aside from the results of the previous section, there is no known classification of which pattern sets admit finite enumeration schemes.  In this section we present empirical results obtained from the implementation of the above algorithms in the Maple package gVATTER.  We will say that a set of forbidden patterns $B$ is \emph{$(d,M)$-scheme countable}, or $(d,M)$-SC, if either $B$, $B^{r}$, or $B^{-1}$ 
admits a finite enumeration scheme of depth at most $d$ with basis gap vectors with norm at most $M$.  As discussed in the introduction, $B$ is $(d,M)$-SC if and only if its set of complement patterns $B^c$ is $(d,M)$-SC.

The following data were assembled by checking whether each vincular pattern $(\sigma, X)$ is $(5,2)$-SC, where we chose $5$ and $2$ as a practical computational considerations.  While there are $k! \cdot 2^{k-1}$ vincular patterns of length $k$, we took advantage of symmetry when able to reduce the number of patterns to check.  To refine analysis, we separated the patterns of length $k$ by the locations of their dashes.  These are represented in Table \ref{tab:SingleSuccess} by ``block type,'' which is a vector describing the number of letters between each dash.  For example, the block type of the pattern $12\d35\d467$ is $(2,2,3)$.

\input{SuccessTableSingle}

It would appear that the success rate is not solely dependent on the number of dashes.  For example, of the 20 pattern classes with a single dash, the five which are not $(5,2)$-SC are all of block type $(2,2)$.  Of the classes with two or more dashes, the most successful block type is $(1,2,1)$ where the dashes do not follow one another.  

In the classical case, schemes were most successful when avoiding multiple patterns simultaneously.  Table \ref{tab:MultSuccess} lists the success rates for finding sets $B$ which are $(5,2)$-SC, for various $B\subseteq \S{2}\cup\S{3}\cup\S{4}\cup\S{5}$.  In the leftmost column, the ``set type'' of a set $B$ refers to the multiset $\{|\sigma|: \sigma\in B\}$.  

\input{SuccessTableMult}

\subsection{Wilf-classification of Vincular Patterns}

We now present some preliminary Wilf-classification results based on the data generated by the schemes.  Two patterns $\sigma, \tau$ are said to be \emph{Wilf-equivalent} if $\sav{n}{\sigma} = \sav{n}{\tau}$ for all $n$, and we denote this $\sigma \we \tau$.  Claesson enumerates permutations avoiding a length 3 pattern with one dash in \cite{Claesson2001}, and Elizalde and Noy enumerate permutations avoiding length 3 patterns with no dashes (i.e., consecutive) in \cite{Elizalde2003}.  Thus we turn our attention to length 4 patterns.  All patterns of length 4 with finite schemes of depth at most 5 are listed in Table \ref{Length4}.  Solid black lines separate classes whose sequences are observed to diverge before the 31st term.

\input{Length4}

The first steps towards general results were taken in \cite{Elizalde2006, Kitaev2005}, which we summarize below in Proposition \ref{SergiPairs}:

\begin{proposition}[Elizalde \cite{Elizalde2006}, Kitaev \cite{Kitaev2005}]\label{SergiPairs}
 Suppose $\sigma, \tau$ are Wilf-equivalent consecutive patterns of length $k$.  Then the following are also Wilf-equivalent:
 \begin{itemize}
  \item $\sigma\d(k+1) \we \tau\d(k+1)$
  \item $\sigma\d(k+2)(k+1) \we \tau\d(k+2)(k+1)$
  \item $\sigma\d(k+2)(k+1) \we \sigma\d(k+1)(k+2)$
  \item $1\d(\sigma_1+1)(\sigma_2+1)\cdots(\sigma_k+1)\d(k+1) \we 1\d(\tau_1+1)(\tau_2+1)\cdots(\tau_k+1)\d(k+1)$
 \end{itemize}
\end{proposition}

 It is clear from symmetries of the square that the following consecutive patterns are Wilf-equivalent:
 \begin{itemize}
  \item $12 \we 21$
  \item $123 \we 321$
  \item $132 \we 213 \we 231 \we 312$
 \end{itemize}
 Thus we see that Proposition \ref{SergiPairs} gets us many of the equivalences which appear in Table \ref{Length4}.  There remain many conjectured pairs, which we summarize below in Conjecture \ref{ConjecturedPairs}.  Note that in each case, the conjectured equivalence is confirmed computationally for permutations of length $n\leq 30$.

\begin{conjecture}\label{ConjecturedPairs}
 We have the following Wilf-equivalences:
 \begin{description}
  \item[(a)] $132\d4 \we 142\d3 \we 241\d3$      
  \item[(b)] $124\d3 \we 421\d3$      
  \item[(c)] $12\d3\d4 \we 12\d4\d3 \we 21\d3\d4 \we 21\d4\d3$  
  \item[(d)] $1\d24\d3 \we 1\d42\d3$  
  \item[(e)] $1\d43\d2 \we 1\d23\d4 \we 1\d34\d2 \we 1\d32\d4$ 
 \end{description}
\end{conjecture}

The equivalences in (a) and (b) are proven in an upcoming paper by the first author \cite{Baxter2012umbral}.  The equivalences in (c) are proven in a different upcoming paper by the first author \cite{Baxter2012swe}.

\subsection{Refinement According to the Inversion Number}

In \cite{Baxter2010}, the first author demonstrates that schemes for permutations avoiding classical patterns can be adapted to compute the number of such permutations with $k$ inversions.  The same refinement applies to the schemes developed above, since they use the same deletion maps $d_R$.

As an example, consider the permutations avoiding $1\d32$.  In Proposition 3 of \cite{Claesson2001}, Claesson shows that the Bell numbers enumerate $1\d32$-avoiding permutations via a bijection, where the number of descents in the permutation is one more than the number of blocks in the corresponding set partition.  We now find the number of permutations with $k$ inversions avoiding $1\d32$, leading to a new refinement of the Bell numbers which is shown in Table \ref{Refined1-32}.  Further terms of the sequence can be generated via the enumeration scheme for $1\d32$ and applying the refinements discussed in \cite{Baxter2010}.\footnote{This refinement is implemented in the Maple package \textsc{gVatter} as \texttt{qMiklos}.}

\begin{table}
	\centering
		\begin{tabular}{r|rcccccccccc}
			      & $k=0$ & 1 & 2 & 3 & 4 & 5 & 6 & 7 & 8 & 9 & 10 \\ \hline
		 $n=0$ & 1 &  &  &  &  &  &  &  &  &  & \\
			   1 & 1 &  &  &  &  &  &  &  &  &  & \\
			   2 & 1 & 1&  &  &  &  &  &  &  &  & \\
			   3 & 1 & 1& 2& 1&  &  &  &  &  &  & \\
			   4 & 1 & 1& 2& 4& 3& 3& 1&  &  &  & \\
			   5 & 1 & 1& 2& 4& 7& 8& 9& 9& 6& 4& 1 \\
		\end{tabular}
	\caption{Number of permutations avoiding $1\dcap32$ of length $n$ with $k$ inversions}
	\label{Refined1-32}
\end{table}

It is interesting to note that if one continues this chart, the columns are each eventually constant.  Specifically, if $f(n,k)$ is the number of permutations avoiding $1\d32$ of length $n$ with $k$ inversions, then $f(n,k)=f(k+1,k)$ for all $n\geq k+1$.   The stagnation can be seen as follows.  
Suppose $\pi$ is a permutation with $k$ inversions and length $n\geq k+2$ which avoids $1\d32$.  We will show that $\pi$ must have $\pi_j = j$ for $j\geq k+2$, which implies $f(n,k) = f(k+1,k)$ for $n \geq k+1$.
First observe that $\pi$ must end with an ascending run, that is, $\pi_{k+1} < \pi_{k+2} < \dotsb < \pi_{n}$.  To see this, suppose for contradiction that $\pi_j > \pi_{j+1}$ for $j \geq k+1$.  Then $\pi_i > \pi_{j+1}$ for each $i<j$ or else $\pi_i \pi_j \pi_{j+1}$ forms a copy of $1\d32$, but this implies $\pi_{j+1}$ is involved in $\j\geq k+1$ inversions which contradicts the assumption that $\pi$ has only $k$ inversions total.  Next suppose $\pi_{j}<j$ for some $j\geq k+2$, and without loss of generality assume this is the minimal such $j$.  Then $\pi_{j-1} > \pi_j$, so $\pi_j =1$ or else $1\pi_{j-1}\pi_j$ forms a copy of $1\d32$.   This implies that $\pi_j$ is involved in $j-1 \geq k+1$ inversions, however, which contradicts our assumption.  Thus we have shown for each $k+2 \leq j \leq n$ that $j \leq \pi_{j} < \pi_{j+1}$ , and so it follows that $\pi_j = j$.

The sequence of $\bigl\{\lim\limits_{n\to \infty} f(n,k)\bigr\}_{k\geq0} $ is given by $1$, $1$, $2$, $4$, $7$, $13$, $22$, $38$, $63$, $105,\dotsc$, which was new to the OEIS \cite{OEIS} and has been added as A188920.  These numbers ought to describe set partitions where a certain statistic is $k$, but it is unclear whether such a statistic would be ``natural.''  It is interesting to note that when one considers the analogous question of permutations avoiding $1\d3\d2$ of length $n$ and $k$ inversions, one also sees that as $n\to\infty$ the number is eventually constant at the number of integer partitions of $k$.  See \cite{Furlinger1985} for a proof in terms of lattice paths or \cite{Claesson2011} for a more recent treatment.

\section{Conclusions and Future Directions}

This paper develops automatable methods to compute $\sav{n}{B}$ for many sets of vincular patterns $B$.  This was accomplished by extending the enumeration schemes developed by Vatter and Zeilberger in \cite{Zeilberger1998, Vatter2008, Zeilberger2006}.  The restrictions on adjacencies which vincular patterns present introduced complications when discovering gap vectors and reversibly deletable sets.   Theorem \ref{thm:GapTest} demonstrates that gap vectors can only be discovered when prefixes are long enough to contain a large portion of a vincular pattern.  Section \ref{AutoDiscRD} explains how the discovery of reversibly deletable sets requires two tests rather than one as well as the introduction of a ``null'' character.  The Maple implementation in \textsc{gVatter}\footnote{Available for download from the both authors' homepages.}, provides a practical tool to compute many terms of $\sav{n}{B}$ if $B$ admits a finite scheme of reasonable depth.

Despite the added complications, Theorem \ref{thm:BigClass} proves that any pattern set containing only consecutive patterns and patterns of the form $\sigma_1 \sigma_2 \cdots \sigma_{t}\d\sigma_{t+1}$ admits a finite scheme.  Hence enumeration schemes may be added to the list of methods to analyze problems in consecutive pattern avoidance.  Classical patterns admitting a finite scheme have not been classified, and there have been few results about infinite classes of patterns which admit finite schemes.

In Question 9.2 of \cite{Vatter2008}, Vatter asks ``Is every sequence produced by a finite enumeration scheme holonomic (i.e., $P$-recursive)?''  It is noted above that $23\d1$ has a finite scheme which produces the Bell numbers, and Sagan shows in \cite{Sagan2010} that the Bell numbers are not holonomic.  Hence a finite scheme can produce a sequence which is not holonomic, although it remains to be seen whether a finite scheme resulting from \emph{classical} patterns (the original context for Vatter's schemes) can yield a non-holonomic sequence.  It should be noted that there is no set $B$ of classical patterns for which $\sav{n}{B}$ is known to be non-holonomic, even when $B$ has no known finite enumeration scheme.

Enumeration schemes provide powerful tools for generating terms of the sequence $\sav{n}{B}$ for a broad class of sets $B$.  Thus far they have been developed for vincular patterns and barred patterns.  This project originated as an attempt to develop enumeration schemes for the bivincular patterns introduced in \cite{Bousquet2010}, but it quickly became apparent that the maps $d_R$ wreak havoc on vertical adjacencies among letters following the prefix and would be unsuitable.\footnote{Deleting a letter at the start of the permutation can create a vertical adjacency at the end of the permutation. For example, to contain the bivincular pattern $\sigma=(12,\{1\},\{1\})$, a permutation $\pi$ must have a subsequence $\pi_{i}\pi_{i+1}$ such that $\pi_{i}+1=\pi_{i+1}$.  The permutation $13524$ avoids $\sigma$ while $d_{\{2\}}(13524)=1423$ contains $\sigma$ in the last two letters.  }  A different recursive structure for $\S{n}$ would need to be exploited to make enumeration schemes work for bivincular patterns which have non-trivial adjacency requirements.

One may wonder whether these methods may be extended to compute the number of permutations of length $n$ admitting $r$ copies of a given pattern, as done in \cite{Noonan1996}.  The problem considered in this paper would be the $r=0$ case.  Since a given letter can be involved in multiple copies of a pattern, the existence of bijective deletion maps $d_R$ becomes far less likely for short prefixes however.  Thus it seems unlikely that the methods contained here will extend easily to the multiple-copies case.

In \cite{Pudwell2010a} the second author extends enumeration schemes to pattern avoidance by words, i.e., permutations of multisets.  The techniques above should extend in a straightforward manner to this case.  Similarly, schemes could be developed to handle permutations (or even words) avoiding barred vincular patterns by combining the techniques of this paper and \cite{Pudwell2010b}.  These extensions have yet to be implemented by computer.

\bibliographystyle{plain}
\bibliography{ESbib}

\end{document}

%% file: macros.tex
\usepackage{amsfonts}
\usepackage{amsmath}
\usepackage{amssymb}
\usepackage{amsthm}
\usepackage{graphicx}
\usepackage{url}
\usepackage{arydshln}
\usepackage[latin1]{inputenc}
\usepackage[numbers,square,sort&compress]{natbib}

\renewcommand{\S}[1]{\ensuremath{\mathcal{S}_{#1}}} 

\newcommand{\Sav}[2]{\ensuremath{\mathcal{S}_{#1}(#2)}} 

\newcommand{\sav}[2]{\ensuremath{s_{#1}(#2)}} 

\newcommand{\Sp}[3]{\ensuremath{\mathcal{S}_{#1}(#2)[#3]}}
\renewcommand{\sp}[3]{\ensuremath{s_{#1}(#2)[#3]}}
\newcommand{\Spt}[4]{\ensuremath{\mathcal{S}_{#1}(#2)[#3;#4]}}
\newcommand{\spt}[4]{\ensuremath{s_{#1}(#2)[#3;#4]}}

   \newcommand{\oi}{\sim}
   \newcommand{\we}{\equiv}

\newcommand{\red}{\mathrm{red}}

\newtheorem{theorem}{Theorem}
\newtheorem{definition}[theorem]{Definition}
\newtheorem{conjecture}[theorem]{Conjecture}
\newtheorem{lemma}[theorem]{Lemma}

\newtheorem{proposition}[theorem]{Proposition}
\newtheorem{algorithm}[theorem]{Algorithm}

\renewcommand{\d}{\makebox[1.1ex]{\rule[.58ex]{.71ex}{.15ex}}}
\newcommand{\dcap}{\!\!-\!\!}
\newcommand{\n}{\!\bullet\!}


%% file: SuccessTableSingle.tex
\begin{table}[htbp]
	\centering
		\begin{tabular}{|r|r|r|r|}
		\hline
			Block type & \parbox{1.2 in}{Number of trivial\\ symmetry classes} & \parbox{1 in}{Number of\\ $(5,2)$-SC classes} & Percentage \\
			\hline
        (2) &  1 &  1 & 100\% \\
      (1,1) &  1 &  1 & 100\% \\
      \hline
			  (3) &  2 &  2 & 100\% \\
			(2,1) &  3 &  3 & 100\% \\
		(1,1,1) &  2 &  2 & 100\% \\
		\hline
		    (4) &  8 &  8 & 100\% \\
		  (3,1) & 12 & 12 & 100\% \\
		  (2,2) &  8 &  3 & 37.5\% \\
		(2,1,1) & 12 &  4 & 25\% \\  
		(1,2,1) &  8 &  6 & 75\% \\
  (1,1,1,1) &  7 &  2 & 28.6\% \\
  \hline
		\end{tabular}
	\caption{Success rate by block type}
	\label{tab:SingleSuccess}
\end{table}

%% file: SuccessTableMult.tex
\begin{table}
	\centering
		\begin{tabular}{|r|r|r|r|}
		\hline
			Set type & \parbox{1.2 in}{Number of trivial\\ symmetry classes} & \parbox{1 in}{Number of\\ $(5,2)$-SC classes} & Percentage \\
			\hline
			\{2\}        &      2         &	   2& 100.0\% \\
			\{2, 2\}     &   		3					&	   3& 100.0\% \\
			\{2, 3\}     &     11         &	  11& 100.0\% \\
      \hline
			\{3\}        &      7         &	   7& 100.0\% \\
			\{3, 3\}     &     70         &	  68&  97.1\% \\
			\{3, 3, 3\}  &    358         &	 354&  98.9\% \\
			\hline
			\{4\}        &     55         &	  35&  63.6\% \\
		  \{4, 4\}     &   4624         &	1600&  34.6\% \\  
			\hline
			\{5\}        &   	479 				&	 144&  30.1\% \\
			\hline
			\{3, 4\}     &   	914   			&	 639&  69.9\% \\
			\{3, 5\}     &   7411         & 2465&  33.3\% \\
			\hline
		\end{tabular}
	\caption{Success rate for sets of patterns, $B$.}
	\label{tab:MultSuccess}
\end{table}

%% file: Length4.tex
\begin{table}
	\centering
		\begin{tabular}{|c|p{2 in}|c|p{2 in}|}
		  \hline
			$\sigma$ & $\{\Sav{n}{\sigma}\}_{n}$ & OEIS \cite{OEIS}& Comments \\
			\hline
			123-4 & $1, 2, 6, 23, 108, 598, 3815, 27532,$ $221708, 197025, \ldots$ & A071076 & \\
			321-4 & & & $\equiv 123\d4$ by Proposition \ref{SergiPairs} \\
			\hline
			132-4 & $1, 2, 6, 23, 107, 585, 3671, 25986,$ $204738, 1776327, \ldots$ & A071075& \\
			231-4 & & & $\we 132\d4$ by Proposition \ref{SergiPairs}\\
			312-4 & & & $\we 132\d4$ by Proposition \ref{SergiPairs}\\
			213-4 & & & $\we 132\d4$ by Proposition \ref{SergiPairs}\\
			142-3 & & & $\we 132\d4$ by Conjecture \ref{ConjecturedPairs}\\
			241-3 & & & $\we 132\d4$ by Conjecture \ref{ConjecturedPairs}\\
			\hline
			124-3 & $1, 2, 6, 23, 107, 584, 3660, 25910,$ $204564, 1782520,\ldots$ & New & \\
			421-3 & & & $\we 124\d3$ by Conjecture \ref{ConjecturedPairs} \\
			\hline
			143-2 & $1, 2, 6, 23, 107, 582, 3622, 25369,$ $197523, 1692535,\ldots$ & New & \\
			\hline
			214-3 & $1, 2, 6, 23, 107, 583, 3637, 25548,$ $199506, 1714383,\ldots$ & New & \\
			\hline
			12-34 & $1, 2, 6, 23, 107, 585, 3669, 25932,$ $203768, 1761109,\ldots$ & A113226  & \\
			12-43 & & & $\we 12\d34$ by Proposition \ref{SergiPairs} \\
			21-43 & & & $\we 21\d43$ by Proposition \ref{SergiPairs} \\
			\hline
      1-24-3 & $1, 2, 6, 23, 104, 532, 3004, 18426,$ $121393, 851810,\ldots $ & A137538 & Wilf equivalent to $25\bar{1}34$ \\
      1-42-3 & & & $\we 1\d24\d3$ by Conjecture \ref{ConjecturedPairs}\\
      \hline
      1-23-4 & $1, 2, 6, 23, 105, 549, 3207, 20577,$ $143239, 1071704,\ldots$ & A113227 &  \\
      1-32-4 & & & $\we 1\d23\d4$ by Proposition \ref{SergiPairs} \\
      1-34-2 & & & $\we 1\d23\d4$ by Conjecture \ref{ConjecturedPairs}\\
      1-43-2 & & & $\we 1\d23\d4$ by Conjecture \ref{ConjecturedPairs}\\
      \hline
      12-3-4 & $1, 2, 6, 23, 105, 550, 3228, 20878,$ $146994, 1116000,\ldots$ & New & \\
      12-4-3 & & & $\we 12\d3\d4$ by Conjecture \ref{ConjecturedPairs}\\
      21-3-4 & & & $\we 12\d3\d4$ by Conjecture \ref{ConjecturedPairs}\\ 
      21-4-3 & & & $\we 12\d3\d4$ by Conjecture \ref{ConjecturedPairs}\\
      \hline
		\end{tabular}
	\caption{Dashed patterns of length 4 admitting schemes of depth at most 5}
	\label{Length4}
\end{table}